\documentclass[11pt]{article}
\usepackage{algorithm, algpseudocode}
\usepackage{amsmath,amssymb,amsthm}
\usepackage{verbatim}
\usepackage{mathrsfs}
\usepackage{subfig}
\usepackage{color} 
\usepackage{graphicx}
 \vspace{0.6cm}

\newtheorem{lemma}{Lemma}[section]

\newtheorem{them}[lemma]{Theorem}
\newtheorem{lm}[lemma]{Lemma}
\newtheorem{coro}[lemma]{Corollary}

\newtheorem{rem}{Remark}
\newtheorem{claim}{Claim}
\newtheorem{exmp}{Example}

\newtheorem{que}{Question}
\def \T{\textup{T}}
\begin{document}
	\title{Generalized spectral characterization of signed trees}
\author{\small Yizhe Ji$^{{\rm a}}$\quad\quad Wei Wang$^{\rm a}$\thanks{Corresponding author: wang\_weiw@xjtu.edu.cn}\quad\quad Hao Zhang$^{\rm b}$
\\
{\footnotesize$^{\rm a}$School of Mathematics and Statistics, Xi'an Jiaotong University, Xi'an 710049, P. R. China}\\
{\footnotesize$^{\rm b}$School of Mathematics, Hunan University, Changsha 410082, P. R. China}
}
\date{}
	\maketitle
	\begin{abstract}
Let $T$ be a tree with an irreducible characteristic polynomial $\phi(x)$ over $\mathbb{Q}$. Let $\Delta(T)$ be the discriminant of $\phi(x)$. It is proved that if $2^{-\frac n2}\sqrt{\Delta(T)}$ (which is always an integer) is odd and square free, then every signed tree with underlying graph $T$
is determined by its generalized spectrum.
	\end{abstract}
\noindent\textbf{Keywords:} Graph spectra; Cospectral graphs; Determined by spectrum; Rational orthogonal matrix; Signed graph\\
\noindent\textbf{Mathematics Subject Classification:} 05C50
	\section{Introduction}



It is well known that the spectra of graphs encode a lot of combinatorial information about the given graphs. A major unsolved question in spectral graph theory is: ``What kinds of graphs are determined (up to isomorphism) by their spectrum (DS for short)?". The problem originates from chemistry and was raised in 1956 by G\"{u}nthard and Primas~\cite{[1]}, which relates H\"{u}ckle's theory in chemistry to graph spectra. The above problem is also closely related to a famous problem of Kac~\cite{[13]}: ``Can one hear the shape of a drum?" Fisher~\cite{[12]} modelled the drum by a graph, and the frequency of the sound was characterized by the eigenvalues of the graph. Hence, the two problems are essentially the same.
	
	It was commonly believed that every graph is DS until the first counterexample (a pair of cospectral but non-isomorphic trees) was found by Collatz and Sinogowitz~\cite{[2]} in 1957. Another famous result on cospectral graphs was given by Schwenk~\cite{[3]}, which states that almost every tree is not DS. For more constructions of cospectral graphs, see, e.g.,~\cite{LS,GM1,S}. However, it turns out that showing a given graph to be DS is generally very hard and challenging. Up to now, only a few graphs with very special structures are known to be DS. We refer the reader to~\cite{[8],[9]} for more background and known results.
	
	In recent years, Wang and Xu~\cite{WX1} and Wang~\cite{Wang1,Wang2} considered a variant of the above problem. For a simple graph $G$, they defined the \emph{generalized spectrum} of $G$ as the spectrum of $G$ together with that of its complement $\bar G$. A graph $G$ is said to be \emph{determined by its generalized spectrum} (DGS for short), if any graph having the same generalized spectrum as $G$ is necessarily isomorphic to $G$.
	
	Let $G$ be a graph on $n$ vertices with adjacency matrix $A = A(G)$. The \emph{walk-matrix} of $G$ is defined as $$W(G)=[e,Ae,\ldots,A^{n-1}e],$$ where $e$ is the all-one vector. Wang~\cite{Wang1,Wang2} proved the following theorem.	

\begin{them}[\cite{Wang1,Wang2}]\label{wang1}
		If $2^{-\lfloor\frac n2\rfloor}\det (W)$ is odd and square-free, then $G$ is DGS.
	\end{them}

The problem of spectral determination of ordinary graphs naturally extends to signed graphs. This paper is a continuation along this line of research for signed graphs in the flavour of Theorem~\ref{wang1}.

Let $\Delta(T)$ be the discriminant of a tree $T$ (see Section 4 for the definition). The main result of the paper is the following theorem.
\begin{them}\label{main}
		Let $T$ be a tree on $n$ vertices with an irreducible characteristic polynomial over $\mathbb{Q}$. If $2^{-\frac n2}\sqrt{\Delta(T)}$ (which is always an integer) is odd and square free, then every signed tree with underlying graph $T$ is DGS.
	\end{them}

As an immediately consequence of Theorem~\ref{main}, we have
\begin{coro}
	Let $T$ and $T'$ be two cospectral and non-isomorphic trees with a common irreducible characteristic polynomial. Suppose $2^{-\frac n2}\sqrt{\Delta(T)}$ is odd and square free. Then no two signed trees with underlying graphs $T$ and $T'$ respectively are generalized cospectral.
\end{coro}

\begin{exmp} \textup{Let $T$ and $T'$ be two cospectral non-isomorphic trees on 14 vertices (see Fig.~1) with a common irreducible characteristic polynomial
$$\phi(T)=\phi(T')=-1 + 16 x^2 - 79 x^4 + 157 x^6 - 143 x^8 + 63 x^{10} - 13 x^{12} + x^{14}.$$ It can be easily computed that $2^{-7}\sqrt{\Delta(T)}=2^{-7}\sqrt{\Delta(T')}=5\times11\times 4754599$, which
is odd and square-free. Thus, according to Theorem~\ref{main}, every signed tree with underlying graph $T$ (resp. $T'$) are DGS. In particular, no two signed trees with underlying graphs $T$ and $T'$ respectively are generalized cospectral.}
\begin{figure}
	\centering
	\includegraphics[width=.4\textwidth]{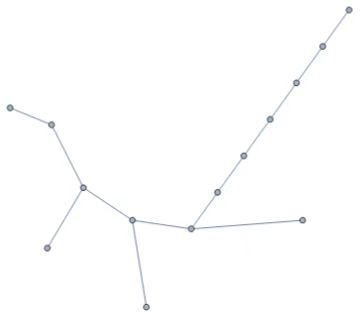}
	\includegraphics[width=.4\textwidth]{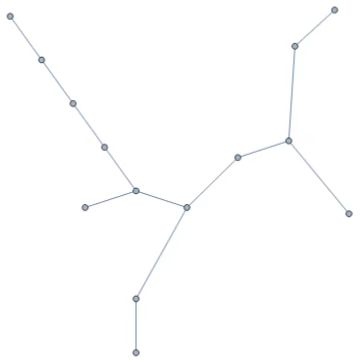} 
	\caption{A pair of cospectral non-isomorphic trees on 14 vertices} 
	\label{img}
\end{figure}
\end{exmp}
	
	Theorem~\ref{main} shows that whenever the underlying tree $T$ with $n$ vertices satisfies a simple arithmetic condition, then all the $2^{n-1}$ signed trees (including $T$ itself) whose underlying is $T$ is DGS. That is, the DGS property of all these signed trees only depends on the underlying graph $T$. This is somewhat unexpected, since given a pair of trees $T$ and $T'$, it seems time consuming even to check whether there exist two signed trees with underlying graphs $T$ and $T'$ respectively that are generalized cospectral; see Example 1.

We mention that Theorem~\ref{main} is the best possible in the sense that it is no longer true if $2^{-\frac n2}\sqrt{\Delta(T)}$ has a multiple odd prime factor. Moreover, the irreducibility assumption of the characteristic polynomial of the tree is essential which cannot be removed; see Remarks 1 and 2 in Section 4.

The rest of the paper is organized as follows. In Section 2, we give some preliminary results that will be needed in the proof of Theorem~\ref{main}. In Section 3, we give a structure theorem, which plays a key role in the paper. In Section 4, we present the proof of~Theorem \ref{main}. Conclusions and future work are given in Section 5.

\section{Preliminaries}


For the convenience of the reader, we give some preliminary results that will be needed later in the paper. For more results in spectral graphs theory, we refer to~\cite{BH,CDS}

 Let $G=(V,E)$ be a simple graph. A \emph{signed graph} is a graph obtained from $G$ by assigning a sign $1$ or $-1$ to every edge according to a mapping $\sigma: E\rightarrow\{1,-1\}$. We use $\Gamma=(G,\sigma)$ to denote a signed graph with \emph{underlying graph} $G$ and sign function (signature) $\sigma$.
 We call a signed graph a \emph{signed bipartite graph}, if its underlying graph is bipartite.

    Let $U$ be a subset of $V$ such that $(U,V\setminus U)$ is a partition of $V$. A \emph{switching} w.r.t. $U$ (or $V \setminus U$) is an operation that changes all the signs of edges between $U$ and $V\setminus U$, while keeps the others unchanged.
         Two signed graphs $\Gamma$ and $\Gamma'$ are \emph{switching-equivalent} if $\Gamma'$ can be obtained from $\Gamma$ by a switching operation, or equivalently, there exists a diagonal matrix $D$ with all diagonal entry $\pm 1$ such that $DA(\Gamma)D=A(\Gamma')$.
     A signed graph is \emph{balanced} if every cycle contains an even number of edges with sign -1. It is well-known that a signed graph is balanced
    if and only it is switching equivalent to a unsigned graph.

Let $\Gamma$ be a signed graph with adjacency matrix $A(\Gamma)$. The \emph{characteristic polynomial} of $\Gamma$ is defined as the characteristic polynomial of $A(\Gamma)$, i.e.,
$\phi(\Gamma;x)=\det(x I-A(\Gamma))$, where $I$ is the identity matrix. Two signed graphs $\Gamma$ and $\Gamma'$ with adjacency matrices $A(\Gamma)$ and $A(\Gamma')$ respectively are called \emph{generalized cospectral} if $$\det(xI-A(\Gamma))=\det(xI-A(\Gamma'))~{\rm and}~\det(xI-(J-I-A(\Gamma)))=\det(xI-(J-I-A(\Gamma'))),$$
where $J$ is the all-one matrix and $J-I-A(\Gamma)$ formally denotes the complement of $\Gamma$ (it is indeed the complement of $\Gamma$ if every edge of $\Gamma$ has been assigned a positive sign +1).
A signed graph $\Gamma$ is said to be \emph{determined by the generalized spectrum} (DGS for short), if any signed graph that is generalized cospectral with $\Gamma$ is isomorphic to $\Gamma$.

A polynomial $f(x)\in{\mathbb{Q}[x]}$ is \emph{irreducible} if it cannot be factored into two polynomials with rational coefficients of lower degree. Let $f(x)\in{\mathbb{Q}[x]}$ be an irreducible polynomial with degree $n$ and $\alpha$ be one of its root. Then $\mathbb{Q}(\alpha)=\{c_0+c_1\alpha+\cdots+c_{n-1}\alpha^{n-1}: c_i\in {\mathbb{Q}}, ~0\leq i\leq n-1\}$ is a \emph{number field} which is isomorphic to $\mathbb{Q}[x]/(f(x))$ and is obtained by adding $\alpha$ to $\mathbb{Q}$; see e.g.~\cite{Lang}.

An orthogonal matrix $Q$ is a square matrix such that $Q^{\T}Q=I_n$. It is called \emph{rational} if every entry of $Q$ is a rational number, and \emph{regular} if each row sum of $Q$ is $1$, i.e., $Qe=e$, where $e$ is the all-one column vector. Denote by RO$_{n}(\mathbb{Q})$ the set of all $n$ by $n$ regular orthogonal matrices with rational entries.
	
In 2006, Wang and Xu~\cite{WX1} initiated the study of the generalized spectral characterization of graphs. For two generalized cospectral graphs $G$ and $H$, they obtained the following result (see also~\cite{JN}), which plays a fundamental role in their method.
\begin{them}[\cite{JN},\cite{WX1}]\label{rational} Let $G$ be a  graph. Then there exists a graph $H$ such that $G$ and $H$ are generalized cospectral if and only if there exists a regular orthogonal matrix $Q$ such that
\begin{equation}\label{E1.1}
Q^{\T}A(G)Q=A(H).
\end{equation}
Moreover, if $\det W(G)\neq 0$, then $Q\in \textup{RO}_{n}(\mathbb{Q})$ is unique and $Q=W(G)W^{-1}(H)$.
\end{them}

A graph $G$ with $\det W(G)\neq 0$ is called \emph{controllable} (see~\cite{Godsil}), denoted by $\mathcal{G}_n$ the set of all controllable graphs on $n$ vertices.  For a graph $G\in\mathcal{G}_n$, define
$$ \mathcal{Q}(G):=\{Q\in \textup{RO}_{n}(\mathbb{Q}):~Q^{\T}A(G)Q=A(H)~ {\rm for~ some~ graph}~ H \}.$$
Then according to Theorem~\ref{rational}, it is easy to obtain the following

\begin{them}[\cite{WX1}]\label{permutationdgs}
Let $G$ be a controllable graph. Then $G$ is DGS if and only if the set $\mathcal{Q}(G)$ contains only permutation matrices.
\end{them}

	The above theorems extend naturally to signed graphs. By Theorem~\ref{permutationdgs}, finding out the possible structure of all $Q\in \mathcal{Q}(G)$ is a key to determine whether a (signed) graph $G$ is DGS.\\

\noindent
\textbf{Notations}: We use $e_n$ (or $e$ if there is no confusion arises) to denote an $n$-dimensional column all-one vector, and $J$ the all-one matrix. For a vector $\alpha=(a_1,a_2,\ldots,a_n)^{\T}\in{\mathbb{R}^n}$, we use $||\alpha||_2=(a_1^2+a_2^2+\cdots+a_n^2)^{1/2}$ to denote the Euclidean norm of $\alpha$.

\section{A structure theorem for $Q$}

The key observation of this paper is the following theorem which shows that for two generalized cospectral signed bipartite graphs with a common irreducible characteristic polynomial, the regular rational orthogonal matrix carried out the similarity of their adjacency matrices has a special structure.
	
	\begin{them}\label{lm1}
		Let $\Gamma$ and $\tilde{\Gamma}$ be two generalized cospectral signed bipartite graphs with a common irreducible characteristic polynomial $\phi(x)$ over $\mathbb{Q}$. Suppose that the adjacency matrices of $\Gamma$ and $\tilde{\Gamma}$ are given as follows, respectively:
  $$A=A(\Gamma)=\left[
		\begin{matrix}
		O & M\\
		M^{\rm T} & O
		\end{matrix}
		\right],~\tilde{A}=A(\tilde{\Gamma})=\left[
		\begin{matrix}
		O &  \tilde{M}\\
		\tilde{M}^{\rm T} & O
		\end{matrix}
		\right].$$
Then there exists a regular orthogonal matrix $Q$ such that $Q^{\rm T}AQ=\tilde{A}$, where $$Q=\left[
		\begin{matrix}
		Q_1 & O\\
		O & Q_2
		\end{matrix}
		\right]~ {\rm or} ~Q=\left[
		\begin{matrix}
		O & Q_1\\
		Q_2 & O
		\end{matrix}
		\right]$$
with $Q_1$ and $Q_2$ being regular rational orthogonal matrices, respectively.
	\end{them}
\begin{coro}The matrix $Q$ in Theorem~\ref{lm1} is the unique rational orthogonal matrix such that $Q^{\T}AQ=\tilde{A}$.
\end{coro}
\begin{proof} The irreducibility assumption of the characteristic polynomial of $A$ implies that $\Gamma$ is controllable.
Then the corollary follows immediately from Theorem~\ref{rational}.

\end{proof}

	To give the proof of Theorem \ref{lm1}, we need several lemmas below.

\begin{lm}\label{LL1} Let $\Gamma$ and $\tilde{\Gamma}$ be two generalized cospectral signed graphs with adjacency matrices $A$ and $\tilde{A}$, respectively.
Then $e^{\rm T} (\lambda I-A)^{-1}e=e^{\rm T} (\lambda I-\tilde{A})^{-1}e$.
\begin{proof} It can be easily computed that
 \begin{eqnarray*}
&&\det(\lambda I-(A+t J))\\
&=&\det(\lambda I-A)\det (I-t(\lambda I-A)^{-1}ee^{\rm T})\\
&=&\det(\lambda I-A)(1-t e^{\rm T}(\lambda I-A)^{-1}e).
\end{eqnarray*}
 Similarly, $\det(\lambda I-(\tilde{A}+t J))=\det(\lambda I-\tilde{A})(1-t e^{\rm T}(\lambda I-\tilde{A})^{-1}e)$. Thus, the lemma follows.

\end{proof}

\end{lm}

\begin{lm}[\cite{GR}]\label{LL2} $(\lambda I-A)^{-1}=\sum_{i=1}^{n}\frac{\xi_i\xi_i^{\rm T}}{\lambda-\lambda_i}$, where $\xi_i$'s are normalized
eigenvectors of $A$ associated with $\lambda_i$, for $\leq i\leq n$.

\end{lm}

\begin{lm}[\cite{wang0}]\label{poly} Let $A=(a_{ij})$ be a symmetric integral matrix with an irreducible characteristic polynomial $\phi(x)$.
Let $\lambda_1,\ldots,\lambda_n$ be the distinct eigenvalues of $A$. Then there exist polynomials
$\phi_i(x)\in{\mathbb{Q}[x]}$ with $\deg \phi_i<n$ such that the eigenvectors $\xi_i$ of $A$ associated with $\lambda_i$
can be expressed as
$${\xi}_i=(\phi_1(\lambda_i),\phi_2(\lambda_i),\ldots,\phi_n(\lambda_i))^{\rm T}$$ for $1\leq i\leq n$.
 \end{lm}

\begin{proof}
 Let $\lambda_1$ be an eigenvalue of $A$ with corresponding eigenvector $\xi_1$.
 Consider the linear system of equations $(\lambda_1 I-A)\xi_1=0$. By Gaussian elimination, there exist $x_i\in{\mathbb{Q}(\lambda_1)}$ such that
$\xi_1=(x_1,x_2,\ldots,x_n)^{\T}$.  Note $\mathbb{Q}(\lambda_1)$ is a number field.
There exist polynomials
$\phi_i(x)\in{\mathbb{Q}[x]}$ with $\deg \phi_i<n$ such that $x_i=\phi_i(\lambda_1)$.

By the $k$-th equation of $(\lambda_1 I-A)\xi_1=0$, we have $\psi(\lambda_1):=\sum_{j=1}^na_{k,j}\phi_j(\lambda_1)-\lambda_1\phi_j(\lambda_1)=0$, for $1\leq k\leq n$.
Note $\psi(x)\in{\mathbb{Q}[x]}$ and $\psi(\lambda_1)=0$. By the irreducibility of $\phi(x)$, we have $\phi(x)$ divides $\psi(x)$.
 Thus, $\psi(\lambda_i)=0$ for for $1\leq i\leq n$, and ${\xi}_i=(\phi_1(\lambda_i),\phi_2(\lambda_i),\ldots,\phi_n(\lambda_i))^{\rm T}$ is an eigenvector associated with $\lambda_i$.

\end{proof}

Next, we collect some simple facts about the relationships of eigenvalues/eigenvectors between the adjacency matrix $A$ of a signed bipartite graph $\Gamma$ and its bipartite-adjacency matrix $M$.

\begin{lm}\label{LL4}
		Let $\Gamma$ be a signed bipartite graph with an irreducible characteristic polynomial over $\mathbb{Q}$. Let the adjacency matrix of $\Gamma$ be $A=A(\Gamma)=\left[\begin{matrix}
		O & M \\
		M^{\T}&O
		\end{matrix}\right]$. Suppose that $\left[\begin{matrix}
		u \\
		v
		\end{matrix}\right]$ is an eigenvector of $A$ associated with an eigenvalue $\lambda$.
Then
\begin{enumerate}
\item $\lambda^2$ is an eigenvalue of $MM^{\T}$ and $M^{\T}M$ with corresponding eigenvectors $u$ and $v$, respectively;

\item $u$ and $v$ have the same length, i.e., $||u||_2=||v||_2$;
\item    $\left[\begin{matrix}
		u \\
		-v
		\end{matrix}\right]$ (resp. $\left[\begin{matrix}
		-u \\
		v
		\end{matrix}\right]$) is an eigenvector of $A$ associated with eigenvalue $-\lambda$;

\item The characteristic polynomials of $MM^{\T}$ (resp. $M^{\T}M$) is irreducible over $\mathbb{Q}$.

\end{enumerate}

	\end{lm}

	\begin{proof}
Note that the characteristic polynomial $\phi(x)$ of $A$ is irreducible, it follows that zero can never be an eigenvalue of $A$, and hence $M$ must be a square matrix of order $m:=n/2$.

Let $\lambda\neq 0$ be any eigenvalue of $A$ with corresponding eigenvector
	$\left[\begin{matrix}
		u \\
		v
		\end{matrix}\right]$. Then
		\begin{equation}\label{MM1}
A\left[\begin{matrix}
		u \\
		v
		\end{matrix}\right]=\left[\begin{matrix}
		Mv \\
		M^{\T}u
		\end{matrix}\right]=\lambda\left[\begin{matrix}
		u \\
		v
		\end{matrix}\right]\Longleftrightarrow \left\{\begin{aligned}
			Mv=\lambda u,\\
			M^{\T}u=\lambda v.
\end{aligned}\right.
\end{equation}
		 Thus, we have $u\neq 0$ and $v\neq 0$, for otherwise we would have $u=v=0$, since $\lambda\neq 0$.
		It follows that $$MM^{\T}u=\lambda^2u,~M^{\T}Mv=\lambda^2v.$$

It follows from $Mv=\lambda u$ that $u^{\T}Mv=\lambda u^{\T}u$.
By $M^{\T}u=\lambda v$ we get $v^{\T}M^{\T}u=\lambda v^{\T}v$. Note $u^{\T}Mv=(u^{\T}Mv)^{\T}=v^{\T}M^{\T}u$. It follows that $\lambda u^{\T}u=\lambda v^{\T}v$, and hence
 $ u^{\T}u= v^{\T}v$ since $\lambda\neq 0$.
		
		Note that
		$$A\left[\begin{matrix}
		u \\
		-v
		\end{matrix}\right]=\left[\begin{matrix}
		-Mv \\
		M^{\T}u
		\end{matrix}\right]=-\lambda\left[\begin{matrix}
		u \\
		-v
		\end{matrix}\right].$$
		Hence, $\left[\begin{matrix}
		u \\
		-v
		\end{matrix}\right]$ is an eigenvector of $A$ associated with eigenvalue $-\lambda$. Since the characteristic polynomial of $A$ is irreducible, the set of all the eigenvalues of $A$ can be written as $\{\lambda_1,\lambda_2,\ldots,\lambda_m,-\lambda_1,-\lambda_2,\ldots,\lambda_m\}.$
		
		Hence, the set of all the eigenvalues of $MM^{\T}$ (or $M^{\T}M$) can be write as $\{\lambda_1^2,\lambda_2^2,\ldots,\lambda_m^2\}$. Since $\phi(A;x)=(x^2-\lambda_1^2)\cdots(x^2-\lambda_m^2)$ is irreducible over $\mathbb{Q}$, $\phi(MM^{\T};x)=\phi(M^{\T}M;x)=(x-\lambda_1^2)\cdots(x-\lambda_m^2)$ is also irreducible $\mathbb{Q}$.
	\end{proof}	

Now, we present the proof of Theorem~\ref{lm1}.

\begin{proof}[Proof of Theorem~\ref{lm1}]

 Set $m:=n/2$. By Lemma~\ref{LL4}, let $\lambda_1,\lambda_2,\ldots,\lambda_m,-\lambda_1,-\lambda_2,\ldots,-\lambda_m$ be the eigenvalues of $A$ and $\tilde{A}$ with corresponding normalized eigenvectors
\begin{eqnarray}\label{LP1}
\frac{1}{\sqrt{2}}\left[\begin{matrix}
		u_1\\
		v_1
		\end{matrix}\right],\ldots,\frac{1}{\sqrt{2}}\left[\begin{matrix}
		u_m\\
		v_m
		\end{matrix}\right],\frac{1}{\sqrt{2}}\left[\begin{matrix}
		u_1\\
		-v_1
		\end{matrix}\right],\ldots,\frac{1}{\sqrt{2}}\left[\begin{matrix}
		u_m\\
		-v_m
		\end{matrix}\right],\\
\frac{1}{\sqrt{2}}\left[\begin{matrix}
		\tilde{u}_1\\
		\tilde{v}_1
		\end{matrix}\right],\ldots,\frac{1}{\sqrt{2}}\left[\begin{matrix}
		\tilde{u}_m\\
		\tilde{v}_m
		\end{matrix}\right],\frac{1}{\sqrt{2}}\left[\begin{matrix}
		\tilde{u}_1\\
		-\tilde{v}_1
		\end{matrix}\right],\ldots,\frac{1}{\sqrt{2}}\left[\begin{matrix}
		\tilde{u}_m\\
		-\tilde{v}_m
		\end{matrix}\right],
\end{eqnarray}
respectively, where $u_i,\tilde{u}_i,v_i,\tilde{v}_i\in{\mathbb{R}^n}$ are $m$-dimensional unit vectors.

		By Lemma~\ref{LL1}, we have $e^{\T}(xI-A)^{-1}e=e^{\T}(xI-\tilde A)^{-1}e$. It follows from Lemma~\ref{LL2} that
		{\small \begin{equation}
		\sum_{i=1}^{m} \frac {(\frac{1}{\sqrt{2}}e_{2m}^{\T}\left[\begin{matrix}
			u_i\\
			v_i
			\end{matrix}\right])^2}{x-\lambda_i}+\sum_{i=1}^{m} \frac {(\frac{1}{\sqrt{2}}e_{2m}^{\T}\left[\begin{matrix}
			u_i\\
			-v_i
			\end{matrix}\right])^2}{x+\lambda_i}=\sum_{i=1}^{m} \frac {(\frac{1}{\sqrt{2}}e_{2m}^{\T}\left[\begin{matrix}
			\tilde u_i\\
			\tilde v_i
			\end{matrix}\right])^2}{x-\lambda_i}+\sum_{i=1}^{m} \frac {(\frac{1}{\sqrt{2}}e_{2m}^{\T}\left[\begin{matrix}
			\tilde u_i\\
			-\tilde v_i
			\end{matrix}\right])^2}{x+\lambda_i}.
		\end{equation}}

Hence, we have that for each $1\leq i\leq m$,
		\begin{eqnarray}\label{EQ1}
\left\{\begin{aligned}
			(e_{m}^{\T}u_i+e_m^{\T}v_i)^2&=&(e_{m}^{\T}\tilde u_i+e_m^{\T}\tilde v_i)^2,\\
			(e_{m}^{\T}u_i-e_m^{\T}v_i)^2&=&(e_{m}^{\T}\tilde u_i-e_m^{\T}\tilde v_i)^2.
\end{aligned}\right.
		\end{eqnarray}

For a fixed $i$, we distinguish the following two cases:\\

\noindent
\textbf{Case 1.} $e_{m}^{\T}u_i+e_m^{\T}v_i$ and $e_{m}^{\T}\tilde u_i+e_m^{\T}\tilde v_i$ have the same sign (resp. opposite sign), and $e_{m}^{\T}u_i-e_m^{\T}v_i$ and $e_{m}^{\T}\tilde u_i-e_m^{\T}\tilde v_i$
have the same sign (resp. opposite sign). It follows from~\eqref{EQ1} that
\begin{eqnarray*}
\left\{\begin{aligned}
		e_{m}^{\T}u_i+e_m^{\T}v_i&= &e_{m}^{\T}\tilde u_i+e_m^{\T}\tilde v_i,\\
		e_{m}^{\T}u_i-e_m^{\T}v_i&=&e_{m}^{\T}\tilde u_i-e_m^{\T}\tilde v_i,
\end{aligned}\right.
~{\rm or}~\left\{\begin{aligned}
		e_{m}^{\T}u_i+e_m^{\T}v_i&= &-(e_{m}^{\T}\tilde u_i+e_m^{\T}\tilde v_i),\\
		e_{m}^{\T}u_i-e_m^{\T}v_i&=&-(e_{m}^{\T}\tilde u_i-e_m^{\T}\tilde v_i),
\end{aligned}\right.
		\end{eqnarray*}
		which implies that either i) $e_{m}^{\T}u_i= e_{m}^{\T}\tilde u_i$ and $e_m^{\T}v_i= e_m^{\T}\tilde v_i$; or ii) $e_{m}^{\T}u_i=- e_{m}^{\T}\tilde u_i$ and $e_m^{\T}v_i=- e_m^{\T}\tilde v_i$.

\noindent
\textbf{Case 2.} $e_{m}^{\T}u_i+e_m^{\T}v_i$ and $e_{m}^{\T}\tilde u_i+e_m^{\T}\tilde v_i$ have the same sign (resp. opposite sign), and $e_{m}^{\T}u_i-e_m^{\T}v_i$ and $e_{m}^{\T}\tilde u_i-e_m^{\T}\tilde v_i$ have the opposite sign (resp. same sign). Then
\begin{eqnarray*}
\left\{\begin{aligned}
		e_{m}^{\T}u_i+e_m^{\T}v_i&= &e_{m}^{\T}\tilde u_i+e_m^{\T}\tilde v_i,\\
		e_{m}^{\T}u_i-e_m^{\T}v_i&=&-(e_{m}^{\T}\tilde u_i-e_m^{\T}\tilde v_i),
\end{aligned}\right.
~{\rm or}~\left\{\begin{aligned}
		e_{m}^{\T}u_i+e_m^{\T}v_i&= &-(e_{m}^{\T}\tilde u_i+e_m^{\T}\tilde v_i),\\
		e_{m}^{\T}u_i-e_m^{\T}v_i&=&e_{m}^{\T}\tilde u_i-e_m^{\T}\tilde v_i,
\end{aligned}\right.
		\end{eqnarray*}
		which implies that either i) $e_{m}^{\T}u_i=e_{m}^{\T}\tilde v_i$ and $e_m^{\T}v_i= e_m^{\T}\tilde u_i$; or ii) $e_{m}^{\T}u_i=- e_{m}^{\T}\tilde v_i$ and $e_m^{\T}v_i=- e_m^{\T}\tilde u_i$.

Thus, for a fixed $i$, we may assume that either $e_{m}^{\T}u_i= \tau_ie_{m}^{\T}\tilde u_i$ and $e_m^{\T}v_i= \tau_ie_m^{\T}\tilde v_i$  or $e_{m}^{\T}u_i=\sigma_ie_{m}^{\T}\tilde v_i$ and $e_m^{\T}v_i= \sigma_ie_m^{\T}\tilde u_i$, where $\tau_i,\sigma_i\in \{1,-1\}$.  Next, we show that uniformly, either $e_{m}^{\T}u_i= \tau_ie_{m}^{\T}\tilde u_i$ and $e_m^{\T}v_i= \tau_i e_m^{\T}\tilde v_i$ for all $1\leq i\leq m$ or $e_{m}^{\T}u_i=\sigma_ie_{m}^{\T}\tilde v_i$ and $e_m^{\T}v_i= \sigma_ie_m^{\T}\tilde u_i$ for all $1\leq i\leq m$. This is the key technical part of the proof, which highly depends on the irreducibility assumption of $\phi$.

According to Lemma~\ref{poly}, the eigenvectors of $MM^{\T}$ associated with eigenvalues $\lambda_i^2$ can be expressed as
$\xi_i=(\phi_1(\lambda_i),\phi_2(\lambda_i),\ldots,\phi_m(\lambda_i))^{\T}$, where $\phi_j(x)\in{\mathbb{Q}[x]}$ with $\deg \phi_j<n$.

By Lemma~\ref{LL4}, $u_i$ is an eigenvector of $MM^{\T}$ associated with $\lambda_i^2$. Note $u_i$ is a unit vector. It follows that $u_i$ and $\xi_i/||\xi_i||_2$ differ by at most a sign, i.e.,
there exists a $\epsilon_i\in \{1,-1\}$ such that $u_i=\epsilon_i \frac{\xi_i}{||\xi_i||_2}$, and
\begin{eqnarray*}
v_i&=&\frac{1}{\lambda_i}M^{\T}u_i\\
&=&\frac{\epsilon_i}{\lambda_i}M^{\T}(\phi_1(\lambda_i),\phi_2(\lambda_i),\ldots,\phi_m(\lambda_i))^{\T}/||\xi_i||_2\\
&=&\epsilon_i(\varphi_1(\lambda_i),\varphi_2(\lambda_i),\ldots,\varphi_m(\lambda_i))^{\T}/||\xi_i||_2,
\end{eqnarray*}
for some $\varphi_j(x)\in {\mathbb{Q}[x]}$ with degree less than $n$, for $1\leq j\leq m$. The last equality follows since the entries of the vector $\frac{1}{\lambda_i}M^{\T}(\phi_1(\lambda_i),\phi_2(\lambda_i),\ldots,\phi_m(\lambda_i))^{\T}$ belong to $\mathbb{Q}(\lambda_i)$, which is a number field.
Further note that $||u_i||_2=||v_i||_2=1$, we have $\varphi_1(\lambda_i)^2+\varphi_2(\lambda_i)^2+\cdots+\varphi_m(\lambda_i)^2=||\xi_i||^2,$ for $1\leq i\leq m$.

The above discussions apply similarly to the signed bipartite graph $\tilde{\Gamma}$ with adjacency matrix $\tilde{A}$. Then we have
$\tilde{u}_i=\tilde{\epsilon_i} \frac{\tilde{\xi}_i}{||\tilde{\xi}_i||_2}$ for $\tilde{\epsilon_i}\in\{1,-1\}$, where $\tilde{\xi}_i=(\tilde{\phi_1}(\lambda_i),\tilde{\phi}_2(\lambda_i),\ldots,\tilde{\phi}_m(\lambda_i))^{\T}$, $\tilde{\phi}_j(x)\in{\mathbb{Q}[x]}$ with $\deg \tilde{\phi}_j<n$.
Moreover, $\tilde{v}_i=\tilde{\epsilon}_i(\tilde{\varphi}_1(\lambda_i),\tilde{\varphi}_2(\lambda_i),\ldots,\tilde{\varphi}_m(\lambda_i))^{\T}/||\tilde{\xi}_i||_2$ with $\tilde{\varphi}_j(x)\in {\mathbb{Q}[x]}$ with degree less than $n$, and $\tilde{\varphi}_1(\lambda_i)^2+\tilde{\varphi}_2(\lambda_i)^2+\cdots+\tilde{\varphi}_m(\lambda_i)^2=||\tilde{\xi}_i||^2$.

\begin{claim} If $e_{m}^{\T}u_1=\tau_1e_{m}^{\T}\tilde u_1$ and $e_m^{\T}v_1=\tau_1e_m^{\T}\tilde v_1$ with $\tau_1\in\{1,-1\}$, then $e_{m}^{\T}u_i=\tau_ie_{m}^{\T}\tilde u_i$ and $e_m^{\T}v_i=\tau_i e_m^{\T}\tilde v_i$ for some $\tau_i\in\{1,-1\}$, for all $2\leq i\leq m$.
\end{claim}
\begin{proof}
Actually, it follows from $e_{m}^{\T}u_1=\tau_1e_{m}^{\T}\tilde u_1$ that
\begin{equation}\label{PQ1}
\epsilon_1\frac{\sum_{j=1}^m \phi_j(\lambda_1)}{\sqrt{\sum_{j=1}^m \phi_j^2(\lambda_1)}}=\tau_1\tilde{\epsilon_1}\frac{\sum_{j=1}^m \tilde{\phi}_j(\lambda_1)}{\sqrt{\sum_{j=1}^m \tilde{\phi}_j^2(\lambda_1)}}.
\end{equation}
Taking squares on both sides of \eqref{PQ1}, it follows that $$\Phi(\lambda_1):=(\sum_{j=1}^m \phi_j(\lambda_1))^2\sum_{j=1}^m \tilde{\phi}_j^2(\lambda_1)-(\sum_{j=1}^m \tilde{\phi}_j(\lambda_1))^2\sum_{j=1}^m \phi_j^2(\lambda_1)=0.$$
Note that $\phi(x)$ is irreducible and $\Phi(x)\in{\mathbb{Q}[x]}$. It follows that $\phi(x)\mid \Phi(x)$. Hence $\Phi(\lambda_i)=0$ and $e_{m}^{\T}u_i=\tau_i e_{m}^{\T}\tilde u_i$ for some $\tau_i\in\{1,-1\}$, for $2\leq i\leq m$. Similarly, we have $e_m^{\T}v_i=\tilde{\tau}_ie_m^{\T}\tilde v_i$ for some $\tilde{\tau}_i\in\{1,-1\}$, for $2\leq i\leq m$.
Next, we show that $\tau_i$ and $\tilde{\tau}_i$ coincide, i.e., $\tau_i=\tilde{\tau}_i=\pm 1$, for all $2\leq i\leq m$.

In fact, it follows from $e_m^{\T}v_1=\tau_1e_m^{\T}\tilde v_1$ that
\begin{equation}\label{PQ2}
\epsilon_1\frac{\sum_{j=1}^m \varphi_j(\lambda_1)}{\sqrt{\sum_{j=1}^m \phi_j^2(\lambda_1)}}=\tau_1\tilde{\epsilon_1}\frac{\sum_{j=1}^m \tilde{\varphi}_j(\lambda_1)}{\sqrt{\sum_{j=1}^m \tilde{\phi}_j^2(\lambda_1)}}.
\end{equation}

It is easy to see that all the numerators in Eqs.~\eqref{PQ1} and \eqref{PQ2} are non-zero. For example, if $\sum_{j=1}^m \phi_j(\lambda_1)=0$, then
$\sum_{j=1}^m \phi_j(\lambda_i)=0$ for $1\leq i\leq m$ by the irreducibility of $\phi$. That is, $e_m^{\T}\xi_i=0$ for $1\leq i\leq m$, which is ridiculous since
$\xi_i$ ($1\leq i\leq m$) are eigenvectors of $MM^{\T}$ constituting a basis of $\mathbb{R}^m$.

Eq.~\eqref{PQ2} divides Eq.~\eqref{PQ1}, it follows that
	\begin{equation}\label{LL3}
\frac{\sum_{j=1}^m \phi_j(\lambda_1)}{\sum_{j=1}^m \varphi_j(\lambda_1)}=\frac{\sum_{j=1}^m \tilde{\phi}_j(\lambda_1)}{\sum_{j=1}^m \tilde{\varphi}_j(\lambda_1)},
\end{equation}
or equivalently, $\Psi(\lambda_1):=\sum_{j=1}^m \phi_j(\lambda_1)\sum_{j=1}^m \tilde{\varphi}_j(\lambda_1)-\sum_{j=1}^m \varphi_j(\lambda_1)\sum_{j=1}^m \tilde{\phi}_j(\lambda_1)=0$.
By the irreducibility of $\phi(x)$, we obtain that $\phi(x)\mid \Psi(x)$, and hence $\Psi(\lambda_i)=0$ for $2\leq i\leq m$. So Eq.~\eqref{LL3} still holds if we replace $\lambda_1$ with any $\lambda_i$, i.e.,
\begin{equation}\label{LLL}
\frac{\sum_{j=1}^m \phi_j(\lambda_i)}{\sum_{j=1}^m \varphi_j(\lambda_i)}=\frac{\sum_{j=1}^m \tilde{\phi}_j(\lambda_i)}{\sum_{j=1}^m \tilde{\varphi}_j(\lambda_i)},~{\rm for}~2\leq i\leq m.
\end{equation}

By the previous discussions, we get that
\begin{equation}\label{OO1}
\epsilon_i\frac{\sum_{j=1}^m \phi_j(\lambda_i)}{\sqrt{\sum_{j=1}^m \phi_j^2(\lambda_i)}}=\tau_i\tilde{\epsilon_i}\frac{\sum_{j=1}^m \tilde{\phi}_j(\lambda_i)}{\sqrt{\sum_{j=1}^m \tilde{\phi}_j^2(\lambda_i)}},~{\rm for}~{\rm for}~ 2\leq i\leq m.
\end{equation}
\begin{equation}\label{OO2}
\epsilon_i\frac{\sum_{j=1}^m \varphi_j(\lambda_i)}{\sqrt{\sum_{j=1}^m \phi_j^2(\lambda_i)}}=\tilde{\tau}_i\tilde{\epsilon_i}\frac{\sum_{j=1}^m \tilde{\varphi}_j(\lambda_i)}{\sqrt{\sum_{j=1}^m \tilde{\phi}_j^2(\lambda_i)}},~ {\rm for}~2\leq i\leq m.
\end{equation}
Eq.~\eqref{OO2} divides Eq.~\eqref{OO1}, we obtain $\frac{\sum_{j=1}^m \phi_j(\lambda_i)}{\sum_{j=1}^m \varphi_j(\lambda_i)}=\frac{\tau_i}{\tilde{\tau}_i}\frac{\sum_{j=1}^m \tilde{\phi}_j(\lambda_i)}{\sum_{j=1}^m \tilde{\varphi}_j(\lambda_i)},$ together with Eq.~\eqref{LLL}, we get the conclusion that $\tau_i=\tilde{\tau}_i=\pm 1$ for $2\leq i\leq m$.
\end{proof}

\begin{claim}: If $e_{m}^{\T}u_1=\sigma_1e_{m}^{\T}\tilde v_1$ and $e_m^{\T}v_1=\sigma_1e_m^{\T}\tilde u_1$ with $\sigma_1\in\{1,-1\}$, then $e_{m}^{\T}u_i=\sigma_ie_{m}^{\T}\tilde v_i$ and $e_m^{\T}v_i=\sigma_ie_m^{\T}\tilde u_i$ for some $\sigma_i\in\{1,-1\}$, for all $2\leq i\leq m$.
\end{claim}
\begin{proof}This follows by using the same argument as Claim 1; we omit the details here.
\end{proof}

 Write $$U=[u_1,u_2,\ldots,u_m],~V=[v_1,v_2,\ldots,v_m],$$
$$\tilde{U}=[\tilde u_1,\tilde u_2,\ldots,\tilde u_m],~ \tilde{V}=[\tilde v_1, \tilde v_2, \ldots, \tilde v_m].$$

If the condition of Claim 1 holds, we may replace $u_i$ and $v_i$ with $-u_i$ and $-v_i$ respectively, whenever $\tau_i=-1$ for $1\leq i\leq m$. Then we have $e_{m}^{\T}u_i=e_{m}^{\T}\tilde u_i$ and $e_m^{\T}v_i=e_m^{\T}\tilde v_i$ for $1\leq i\leq m$.
    Let $R=\frac{1}{\sqrt{2}}\left[
		\begin{matrix}
		U & U\\
		V & -V
		\end{matrix}
		\right]$ and $\tilde R=\frac{1}{\sqrt{2}}\left[
		\begin{matrix}
		\tilde U & \tilde U\\
		\tilde V & -\tilde V
		\end{matrix}
	\right].$
Define \begin{equation}\label{Q1}
Q:=R\tilde{R}^{\T}=\left[
		\begin{matrix}
		U\tilde U^{\T} & O\\
		O&V\tilde V^{\T}
		\end{matrix}
		\right].
\end{equation}
Then $Q$ is an orthogonal matrix and $$R^{\T}AR=\tilde R^{\T}\tilde A\tilde R={\rm diag}(\lambda_1,\ldots,\lambda_m,-\lambda_1,\ldots,-\lambda_m).$$
	Thus, $Q^{\T}AQ=\tilde{A}$.
 Next, it remains to show that $Q$ is regular, i.e., $Qe_{2m}=e_{2m}$, which is equivalent to $\tilde U^{\T}e_m=U^{\T}e_m$ and $\tilde V^{\T}e_m=V^{\T}e_m$. That is, $e_m^{\T}u_i=e_m^{\T}\tilde u_i,~e_m^{\T}v_i=e_m^{\T}\tilde v_i,~{\rm for}~1\leq i\leq m,$ which are precisely that we have obtained before, as desired.

If the condition of Claim 2 holds, similarly, we may replace $u_i$ and $v_i$ with $-u_i$ and $-v_i$ respectively, whenever $\sigma_i=-1$. Then $e_{m}^{\T}u_i=e_{m}^{\T}\tilde v_i$ and $e_m^{\T}v_i=e_m^{\T}\tilde u_i$, for $1\leq i\leq m$. Now let $R=\frac{1}{\sqrt{2}}\left[
		\begin{matrix}
		U & U\\
		V & -V
		\end{matrix}
		\right]$ and $\tilde R=\frac{1}{\sqrt{2}}\left[
		\begin{matrix}
		\tilde U & -\tilde U\\
		\tilde V & \tilde V
		\end{matrix}
		\right].$
Define \begin{equation}\label{Q1}
Q:=R\tilde{R}^{\T}=\left[
		\begin{matrix}
		O & U\tilde V^{\T}\\
		V\tilde U^{\T}&O
		\end{matrix}
		\right].
\end{equation}
Then $Q$ is an orthogonal matrix and still $Q^{\T}AQ=\tilde{A}$ holds. Moreover, it is easy to verify that $Qe_{2m}=e_{2m}$. So $Q$ is regular.

The proof is complete.

\end{proof}

\section{Proof of Theorem~\ref{main}}
	In this section, we present the proof of Theorem~\ref{main}.

	Recall that for a monic polynomial $f(x)\in{\mathbb{Z}[x]}$ with degree $n$, the \emph{discriminant} of $f(x)$ is defined as:
	$$\Delta(f)=\prod_{1\leq i<j\leq n}(\alpha_i-\alpha_j)^2,$$
where $\alpha_1,\alpha_2,\ldots,\alpha_n$ are all the roots of $f(x)$.

	Then it is clear that $\Delta(f)$ is always an integer for $f(x)\in{\mathbb{Z}[x]}$, and $\Delta(f)=0$ if and only if $f$ has a multiple root. Define the \emph{discriminant} of a matrix $A$, denoted by $\Delta(A)$, as the discriminant of its characteristic polynomial, i.e., $\Delta(A):=\Delta(\det(xI-A))$. The \emph{discriminant} of a graph $G$, denoted by $\Delta(G)$, is defined to be the discriminant of its adjacency matrix.

 In~\cite{[22]}, Wang and Yu give the following theorem, which is our main tool in proving Theorem~\ref{main}.
 	
	\begin{them}[\cite{[22]}]\label{lm2}
		Let $A$ be a symmetric integral matrix. Suppose there exists a rational orthogonal matrix $Q$ such that
$Q^{\rm T}AQ$ is an integral matrix. If $\Delta(A)$ is odd and square-free, then $Q$ must be a signed permutation matrix.
	\end{them}

	However, Theorem~\ref{lm2} cannot be used directly, since the $\Delta(\Gamma)$ is always a perfect square for a signed bipartite graph $\Gamma$ with an equal size of bipartition, as shown by the following lemma.

\begin{lm}Let $\Gamma$ be a signed bipartite graph with bipartite-adjacency matrix $M$, where $M$ is a square matrix of order $m:=n/2$. Then $\Delta(\Gamma)=2^n\det^2(M)\Delta^2(M^{\rm T}M)$.
\end{lm}
\begin{proof}
Let the eigenvalues of $\Gamma$ be $\pm \lambda_1,\pm\lambda_2,\ldots,\pm\lambda_m.$
 Then the eigenvalues of $M^{\T}M$ are $\lambda_1^2,\lambda_2^2,\ldots,\lambda_m^2.$  So we have
		\begin{eqnarray*}
			\Delta(\Gamma)&=&\prod_{1\leq i<j\leq m}(\lambda_i-\lambda_j)^2\prod_{1\leq i,j\leq m}(\lambda_i+\lambda_j)^2\prod_{1\leq i<j\leq m}(-\lambda_i+\lambda_j)^2\\
			&=&2^n\lambda_1^2\lambda_2^2\cdots\lambda_m^2\prod_{1\leq i<j\leq m}(\lambda_i^2-\lambda_j^2)^4\\
			&=&2^n\det(M^{\T}M)\Delta^2(M^{\T}M)\\
			&=&2^n (\det(M))^2\Delta^2(M^{\rm T}M).
		\end{eqnarray*}
This completes the proof.
\end{proof}

Let $a_0$ be the constant term of the characteristic polynomial of $G$ defined as above. Then $$a_0=(-1)^m\det(M^{\T}M)=(-1)^m\det{^2}(M).$$ Note that for a tree with an irreducible characteristic polynomial $\phi(x)$, the constant term of $\phi(x)$ is always $\pm1$. Thus we have
\begin{coro}Let $T$ be a tree with an irreducible characteristic polynomial. Then $\Delta(T)=2^n\Delta^2(M^{\rm T}M)$.

\end{coro}

Finally, we are ready to present the proof of Theorem~\ref{main}.
	
	\begin{proof}[{Proof of Theorem~\ref{main}}.]
		Let $\tilde{\Gamma}$ be any signed graph that is generalized cospectral with $\Gamma=(T,\sigma)$. We shall show that $\tilde{\Gamma}$ is isomorphic to $\Gamma$.
 Note that $\tilde{\Gamma}$ has the same number of edges as $\Gamma$ and moreover, the assumption that $\phi(\tilde{\Gamma})=\phi(\Gamma)$ is irreducible forces $\tilde{\Gamma}$ to be connected. Thus, $\tilde{\Gamma}$ is signed graph whose underlying graph is a tree (say $\tilde{T}$), and $\tilde{\Gamma}=(\tilde{T},\tilde{\sigma})$.

 Note that both $T$ and $\tilde{T}$ are balanced as signed graphs, we have
$\phi(T)=\phi(\Gamma)$ and $\phi(\tilde{T})=\phi(\tilde{\Gamma})$.
Let $A(\Gamma)=D_1A(T)D_1$ and $A(\tilde{\Gamma})=D_2A(\tilde{T})D_2$, where $D_1$ and $D_2$ are diagonal matrices
whose diagonal entries are $\pm 1$.

By Theorem~\ref{rational}, the fact that $\Gamma$ and $\tilde{\Gamma}$ are generalized cospectral implies that there exists a regular rational orthogonal matrix $Q$ such that
\begin{equation}\label{QQ}
Q^{\rm T} A(\Gamma)Q=A(\tilde{\Gamma}),
\end{equation}
 i.e., $Q^{\rm T}(D_1A(T)D_1)Q=D_2A(\tilde{T})D_2$, which is equivalent to
$\hat{Q}^{\rm T}A(T)\hat{Q}=A(\tilde{T}),$ where $\hat{Q}=D_1QD_2$ is a rational orthogonal matrix.

	 Let  $$A(T)=\left[
		\begin{matrix}
		O & M\\
		M^{\rm T} & O
		\end{matrix}
		\right], A(\tilde{T})=\left[
		\begin{matrix}
		O &  \tilde{M}\\
		\tilde{M}^{\rm T} & O
		\end{matrix}
		\right].$$
By Theorem~\ref{lm1}, assume without loss of generality that $Q=\left[
		\begin{matrix}
		Q_1 & O\\
		O & Q_2
		\end{matrix}
		\right]$ and  $\hat{Q}=\left[
		\begin{matrix}
		\hat{Q}_1 & O\\
		O & \hat{Q}_2
		\end{matrix}
		\right].$
Then we have $\hat{Q}_1^{\rm T}M\hat{Q}_2=\tilde{M}$. It follows that
 $$\hat{Q}_1^{\rm T}MM^{\rm T}\hat{Q}_1=\tilde{M}\tilde{M}^{\rm T}~{\rm  and}~ \hat{Q}_2^{\rm T}M^{\rm T}M\hat{Q}_2=\tilde{M}^{\rm T}\tilde{M}.$$
 Note that $\Delta(M^{\rm T}M)=\Delta(MM^{\rm T})=2^{-n/2}\sqrt{\Delta(T)}$, which is odd and square-free. Thus, according to Theorem~\ref{lm2}, both $\hat{Q}_1$ and $\hat{Q}_2$ are signed permutation matrices. It follows that $Q=D_1\hat{Q}Q_2$ is a signed permutation matrix. Moreover, note that $Q$ is regular. Therefore, $Q$ is a permutation matrix, and by Eq.~\eqref{QQ}, we conclude that $\tilde{\Gamma}$ is isomorphic to $\Gamma$. The proof is complete.

	\end{proof}

	\begin{rem}\textup{The condition of Theorem~\ref{main} is tight in the sense that Theorem~\ref{main} is no longer true if $2^{-\frac n2}\sqrt{\Delta(T)}$ has a multiple
odd prime factor. Let the signed bipartite-adjacency matrices of two signed trees $T$ and $\tilde{T}$ be given as follows, respectively:
		$${\tiny M=\left(
\begin{array}{ccccccccc}
 -1 & 0 & 0 & 0 & 0 & 0 & 0 & 0 & 0 \\
 -1 & -1 & 0 & 0 & 0 & 0 & 0 & 0 & 0 \\
 1 & 0 & -1 & 0 & 0 & 0 & 0 & 0 & 0 \\
 0 & 0 & 1 & -1 & 0 & 0 & 0 & 0 & 0 \\
 0 & -1 & 0 & 0 & -1 & 1 & 0 & 0 & 0 \\
 0 & 0 & 0 & 0 & -1 & 0 & 0 & 0 & 0 \\
 0 & 0 & 0 & 0 & -1 & 0 & -1 & 1 & 0 \\
 0 & 0 & 0 & 0 & 0 & 0 & -1 & 0 & 0 \\
 0 & 0 & 0 & 0 & 0 & -1 & 0 & 0 & -1 \\
\end{array}
\right),\tilde{M}=\left(
\begin{array}{ccccccccc}
 0 & 0 & 0 & 0 & 0 & 0 & 0 & 1 & -1 \\
 0 & 0 & 0 & -1 & -1 & 0 & 1 & 0 & 0 \\
 0 & 0 & 0 & -1 & 0 & 0 & 0 & 0 & 0 \\
 -1 & -1 & 0 & 0 & 0 & 0 & 0 & 0 & 0 \\
 0 & -1 & -1 & 0 & 0 & 0 & 0 & 0 & 0 \\
 -1 & 0 & 0 & 0 & 0 & 0 & 1 & 0 & 0 \\
 0 & 0 & 0 & 0 & 0 & 0 & -1 & 0 & 0 \\
 0 & 0 & 0 & 0 & 0 & 0 & 0 & -1 & 0 \\
 -1 & 0 & 0 & 0 & 0 & -1 & 0 & 0 & 1 \\
\end{array}
\right).	}$$
		Then $$\phi(T)=\phi(\tilde{T})=-1 + 22 x^2 - 162 x^4 + 538 x^6 - 897 x^8 + 809 x^{10} - 410 x^{12} +
 116 x^{14} - 17 x^{16} + x^{18},$$ which is irreducible over $\mathbb{Q}$. However, $2^{-9}\sqrt{\Delta(T)}=7^2\times347\times357175051$, i.e., $2^{-9}\sqrt{\Delta(T)}$ has a multiple factor 7 and the condition of Theorem~\ref{main} is not satisfied. Actually, there indeed exists a regular rational orthogonal matrix $Q\in \mathcal{Q}(G)$ such that $\tilde A=Q^{\T}AQ$, where
		$Q={\rm diag}(Q_1,Q_2)$ and $Q_1$ and $Q_2$ are given as follows respectively.
$${\tiny Q_1=\frac{1}{7}\left(
\begin{array}{ccccccccc}
 -1 & -1 & -2 & -2 & 4 & 3 & 3 & 2 & 1 \\
 -2 & -2 & 3 & 3 & 1 & -1 & -1 & 4 & 2 \\
 2 & 2 & 4 & -3 & -1 & 1 & 1 & 3 & -2 \\
 4 & -3 & 1 & 1 & -2 & 2 & 2 & -1 & 3 \\
 -3 & 4 & 1 & 1 & -2 & 2 & 2 & -1 & 3 \\
 3 & 3 & -1 & -1 & 2 & -2 & -2 & 1 & 4 \\
 1 & 1 & 2 & 2 & 3 & 4 & -3 & -2 & -1 \\
 2 & 2 & -3 & 4 & -1 & 1 & 1 & 3 & -2 \\
 1 & 1 & 2 & 2 & 3 & -3 & 4 & -2 & -1 \\
\end{array}\right),
Q_2=\frac{1}7\left(
\begin{array}{ccccccccc}
 2 & 2 & 4 & -3 & -1 & 1 & 1 & 3 & -2 \\
 2 & 2 & -3 & 4 & -1 & 1 & 1 & 3 & -2 \\
 -2 & -2 & 3 & 3 & 1 & -1 & -1 & 4 & 2 \\
 4 & -3 & 1 & 1 & -2 & 2 & 2 & -1 & 3 \\
 1 & 1 & 2 & 2 & 3 & 4 & -3 & -2 & -1 \\
 -3 & 4 & 1 & 1 & -2 & 2 & 2 & -1 & 3 \\
 3 & 3 & -1 & -1 & 2 & -2 & -2 & 1 & 4 \\
 -1 & -1 & -2 & -2 & 4 & 3 & 3 & 2 & 1 \\
 1 & 1 & 2 & 2 & 3 & -3 & 4 & -2 & -1 \\
\end{array}
\right).} $$\\}
	\end{rem}

\begin{rem} \textup{Theorem~\ref{lm1} does not hold without the assumption that the characteristic polynomial of $\Gamma$ is irreducible over $\mathbb{Q}$, even if $\Gamma$ is controllable.
Let $\Gamma$ and $\tilde{\Gamma}$ be two signed trees with bipartite-adjacency matrices $M$ and $\tilde{M}$ given as follows respectively:
$${\tiny M=\left(
\begin{array}{ccccccccc}
 1 & 0 & 0 & 0 & 0 & 0 & 0 & 0 & 0 \\
 -1 & 1 & 0 & 0 & 0 & 0 & -1 & 0 & 0 \\
 0 & -1 & -1 & 0 & 0 & 0 & 0 & 0 & 0 \\
 0 & 0 & -1 & 1 & 0 & 0 & 0 & 0 & 0 \\
 0 & 1 & 0 & 0 & 1 & -1 & 0 & 0 & 0 \\
 0 & 0 & 0 & 0 & 1 & 0 & 0 & 0 & 0 \\
 0 & 0 & 0 & 0 & 0 & 0 & -1 & 0 & 0 \\
 0 & 0 & 0 & 0 & 0 & 0 & 1 & 1 & 0 \\
 0 & 0 & 0 & 0 & 0 & 0 & 0 & -1 & -1 \\
\end{array}
\right),~
\tilde{M}=\left(
\begin{array}{ccccccccc}
 1 & 0 & 0 & 0 & 0 & 0 & 0 & 0 & 0 \\
 -1 & 1 & -1 & 0 & 0 & 0 & 0 & 0 & 0 \\
 0 & 0 & -1 & 0 & 0 & 0 & 0 & 0 & 0 \\
 0 & 0 & 1 & 1 & 0 & 0 & 0 & 0 & 0 \\
 0 & 0 & 0 & -1 & -1 & 0 & 0 & 0 & 0 \\
 0 & -1 & 0 & 0 & 0 & 1 & 0 & 0 & 0 \\
 0 & 1 & 0 & 0 & 0 & 0 & 1 & 0 & 1 \\
 0 & 0 & 0 & 0 & 0 & -1 & 0 & -1 & 0 \\
 0 & 0 & 0 & 0 & 0 & 0 & 0 & 0 & -1 \\
\end{array}
\right).}$$
${\tiny Q=\frac{1}{5}\left(
\begin{array}{cccccccccccccccccc}
 5 & 0 & 0 & 0 & 0 & 0 & 0 & 0 & 0 & 0 & 0 & 0 & 0 & 0 & 0 & 0 & 0 & 0 \\
 0 & 5 & 0 & 0 & 0 & 0 & 0 & 0 & 0 & 0 & 0 & 0 & 0 & 0 & 0 & 0 & 0 & 0 \\
 0 & 0 & 0 & 0 & 0 & 3 & -2 & 1 & -1 & 0 & 0 & 0 & 0 & 0 & 1 & -1 & 2 & 2 \\
 0 & 0 & 0 & 0 & 0 & -1 & -1 & -2 & 2 & 0 & 0 & 0 & 0 & 0 & 3 & 2 & 1 & 1 \\
 0 & 0 & 0 & 0 & 0 & -2 & 3 & 1 & -1 & 0 & 0 & 0 & 0 & 0 & 1 & -1 & 2 & 2 \\
 0 & 0 & 0 & 0 & 0 & 1 & 1 & 2 & -2 & 0 & 0 & 0 & 0 & 0 & 2 & 3 & -1 & -1 \\
 0 & 0 & 5 & 0 & 0 & 0 & 0 & 0 & 0 & 0 & 0 & 0 & 0 & 0 & 0 & 0 & 0 & 0 \\
 0 & 0 & 0 & 5 & 0 & 0 & 0 & 0 & 0 & 0 & 0 & 0 & 0 & 0 & 0 & 0 & 0 & 0 \\
 0 & 0 & 0 & 0 & 5 & 0 & 0 & 0 & 0 & 0 & 0 & 0 & 0 & 0 & 0 & 0 & 0 & 0 \\
 0 & 0 & 0 & 0 & 0 & 0 & 0 & 0 & 0 & 5 & 0 & 0 & 0 & 0 & 0 & 0 & 0 & 0 \\
 0 & 0 & 0 & 0 & 0 & 0 & 0 & 0 & 0 & 0 & 5 & 0 & 0 & 0 & 0 & 0 & 0 & 0 \\
 0 & 0 & 0 & 0 & 0 & -1 & -1 & 3 & 2 & 0 & 0 & 0 & 0 & 0 & -2 & 2 & 1 & 1 \\
 0 & 0 & 0 & 0 & 0 & 2 & 2 & -1 & 1 & 0 & 0 & 0 & 0 & 0 & -1 & 1 & 3 & -2 \\
 0 & 0 & 0 & 0 & 0 & 2 & 2 & -1 & 1 & 0 & 0 & 0 & 0 & 0 & -1 & 1 & -2 & 3 \\
 0 & 0 & 0 & 0 & 0 & 1 & 1 & 2 & 3 & 0 & 0 & 0 & 0 & 0 & 2 & -2 & -1 & -1 \\
 0 & 0 & 0 & 0 & 0 & 0 & 0 & 0 & 0 & 0 & 0 & 5 & 0 & 0 & 0 & 0 & 0 & 0 \\
 0 & 0 & 0 & 0 & 0 & 0 & 0 & 0 & 0 & 0 & 0 & 0 & 5 & 0 & 0 & 0 & 0 & 0 \\
 0 & 0 & 0 & 0 & 0 & 0 & 0 & 0 & 0 & 0 & 0 & 0 & 0 & 5 & 0 & 0 & 0 & 0 \\
\end{array}
\right).}$\\
It is easy to verify that $$\phi(\Gamma;x)=
(-1 + x) (1 + x) (-1 - x + x^2) (-1 + x + x^2) (1 - 21 x^2 + 95 x^4 -
   119 x^6 + 60 x^8 - 13 x^{10} + x^{12}),$$ which is reducible over $\mathbb{Q}$ and $\Gamma$ is controllable.
   Nevertheless, the unique regular rational orthogonal matrix $Q$ (shown as above) such that $Q^{\T}A(\Gamma)Q=A(\tilde{\Gamma})$ is
  not the form as in Theorem~\ref{lm1}.}
\end{rem}

\section{Conclusions and Future Work}

In this paper, we have given a simple arithmetic condition on a tree $T$ with an irreducible characteristic polynomial, under which every signed tree with underlying graph $T$ is DGS. This is a little bit surprising in contrast with Schwenk's remarkable result stating almost every tree has a cospectral mate.

However, there are several questions remained to be answered. We end the paper by proposing the following questions:

\begin{que}
How can Theorem~\ref{main} be generalized to signed bipartite graphs?
\end{que}
\begin{que}
Is it true that every tree with an irreducible characteristic
polynomial is DGS?

\end{que}
\begin{que}
Is Theorem~\ref{lm1} true for controllable bipartite graphs?
\end{que}

For Question 1, the difficulty lies in the fact that for a signed bipartite graph $\Gamma$, a signed graph $\tilde{\Gamma}$ generalized cospectral with $\Gamma$ is not necessarily bipartite. For Question 2, we know that it is not true for signed trees. For Question 3, we know that it is not true for controllable signed bipartite graphs.
But generally we do not know any single counterexample to Questions 2 and 3. The above questions need further investigations in the future.

\section*{Acknowledgments}
The research of the second author is supported by National Natural Science Foundation of China (Grant Nos.\,11971376 and 12371357) and the
third author is supported by Fundamental Research Funds for the Central Universities (Grant No.\,531118010622).

The authors would like to thank Professor Huiqiu Lin from East China University of Science and Technology for useful discussions.

\end{document}